\documentclass[11pt,a4paper]{article}
\usepackage[T1]{fontenc}
\usepackage{lmodern,amsmath,amssymb,amsthm,mathtools,booktabs}
\usepackage[margin=27mm]{geometry}
\usepackage{microtype}
\usepackage[hidelinks]{hyperref}
\usepackage{xurl}
\newtheorem{theorem}{Theorem}[section]
\newtheorem{lemma}[theorem]{Lemma}
\newtheorem{proposition}[theorem]{Proposition}

\theoremstyle{definition}

\theoremstyle{remark}

\newcommand{\F}{{\mathbb F_3}}
\newcommand{\dc}{\mathsf d}
\newcommand{\Span}{\operatorname{span}}
\newcommand{\supp}{\operatorname{supp}}

\newcommand{\file}[1]{{\small\nolinkurl{#1}}}
\hypersetup{pdftitle={A twelve-term exclusion for the small Davenport constant of E2 times C3 to the r},pdfauthor={Andreas Volkmann}}
\title{A twelve-term exclusion for the small Davenport constant of
\texorpdfstring{\\$E_2\times C_3^r$}{ E2 times C3 to the r}}
\author{Andreas Volkmann}
\date{19 September 2026}
\begin{document}
\maketitle
\begin{abstract}
Let $E_2$ be the extraspecial group of order $3^5$ and exponent three.
For every $r\ge1$, we prove that a product-one-free sequence of length
$2r+11$ over $E_2\times C_3^r$ cannot contain exactly $2r-1$ central
terms. Thus the critical layer with twelve noncentral terms is excluded.
The proof combines a relative moment criterion for abelian normal
subgroups with a finite theorem in a symplectic four-space over
$\mathbb F_3$. Under explicit subspace occupancy bounds, the family
of balanced triangles that can actually be completed to a nonfull
zero-sum block admits edge weights summing to one on every triangle.
A dual cycle argument reduces this assertion to a potential condition
on branching edges. All remaining configurations contain one of eleven
minimal frames; two separately implemented exhaustive checks verify all
their admissible extensions. The unrestricted check has 9\,544\,046
leaves. Complete source code and execution records are supplied.
The exact value of $\dc(E_2\times C_3^r)$ for arbitrary $r$ is not
determined by this result.
\end{abstract}


\section{Introduction and main result}\label{sec:intro}
A sequence over a finite group is an unordered list with repetitions;
all subsequences below are selections of indexed occurrences.
A nonempty subsequence is \emph{product-one} if its terms can be ordered
with product the identity. The small Davenport constant $\dc(G)$ is
the maximum length of a product-one-free sequence over $G$.

For elementary abelian groups, Olson's theorem gives
$\dc(C_p^n)=n(p-1)$~\cite{Olson}.
The value $\dc(H_{27})=6$, where $H_{27}=\operatorname{UT}_3(\F)$,
was obtained computationally by Cziszter, Domokos and
Sz\"oll\H osi~\cite{CDS} and subsequently proved theoretically by
Godara and Sarkar~\cite{GodaraSarkar}.
The formula $\dc(H_{p^3})=3p-3$ for odd primes is proved
in~\cite{VolkmannUniform}. At the prime three, the author's
subsequent result~\cite[Theorem~1.1]{VolkmannH27} gives
\begin{equation}\label{eq:ranktwo}
\dc(H_{27}\times C_3^s)=2s+6\qquad(s\ge0).
\end{equation}
Its proof assigns weights to the eight edge-disjoint zero-sum triangles
in the nonorthogonality graph of $\F^2$.

The same approach encounters a specific obstruction in symplectic
dimension four: a balanced nonorthogonal triangle can have nonzero
sum, which may be cancelled by additional terms orthogonal to the
whole triangle. This obstruction and the target formula
\begin{equation}\label{eq:target}
\dc(E_2\times C_3^r)=2r+10
\end{equation}
were recorded in~\cite[Section~5]{VolkmannH27}.
Here $E_2$ denotes the extraspecial group of order $3^5$ and exponent
three, also denoted $H_2(\F)$.
The lower bound in~\eqref{eq:target} is obtained by taking two copies
of each of four symplectic basis lifts, of a generator of $E_2'$, and
of each of the $r$ central direct-factor generators. Projection modulo
$E_2'$ first forces all the nonderived multiplicities to be zero in a
product-one subsequence; the derived multiplicity is then also zero.

The present result excludes one layer of the critical length $2r+11$.
\begin{theorem}\label{thm:main}
Let $G=E_2\times C_3^r$ with $r\ge1$. There is no product-one-free
sequence $S$ over $G$ satisfying
\[
|S|=2r+11,\qquad |S\cap Z(G)|=2r-1.
\]
\end{theorem}

The main new finite statement concerns the twelve nonzero projections
of the noncentral terms in $G/Z(G)\cong\F^4$.
Under the occupancy bounds established below, the symplectic
edge cochain is a gradient on every edge belonging to at least three
\emph{actually completable} balanced triangles. A dual integral-cycle
argument then supplies all the required triangle weights.
This is stronger than the affine moment membership used in the final
group-theoretic contradiction.

The proof has four finite components: a nine-term line reduction,
seven- and eight-term reductions at an isotropic plane, and the
eleven-frame twelve-term check. Each component is covered by a complete
finite case analysis and two separately implemented exhaustive
programs. The algebraic reductions and their coverage arguments are
given in the text; the programs and full execution records are ancillary
files. We use~\eqref{eq:ranktwo} as a cited theorem. No unproved
nine-outside-vector quadratic moment assertion is used.

Theorem~\ref{thm:main} does not prove~\eqref{eq:target}. For example,
when $r=1$, the critical case of thirteen noncentral terms and no central
terms is outside its scope.

\section{Coordinates and ordering corrections}\label{sec:coordinates}
Let $P=\F^4$ with its standard nondegenerate alternating form
\[
\omega((a_1,a_2,b_1,b_2),(a'_1,a'_2,b'_1,b'_2))
=a_1b'_1-b_1a'_1+a_2b'_2-b_2a'_2.
\]
Write $Z=\F z\oplus\F^r$ and identify $G$ with $P\oplus Z$
under the multiplication
\begin{equation}\label{eq:law}
(q,w)*(q',w')=(q+q',w+w'+\tfrac12\omega(q,q')z).
\end{equation}
Here $1/2=2$ in $\F$, and $\omega$ is extended to the whole space by
making $Z$ its radical. This is a class-two, exponent-three group
with center $Z$ and derived group $\F z$; its nondegenerate factor
is $E_2$. We call $q$ the \emph{active projection} of $(q,w)$.

For terms $x_i$ in such a Lie coordinate model and a set of indices
$B$, let $\Omega(B)$ be the set of scalars
\begin{equation}\label{eq:omega}
\frac12\sum_{j<k}\omega(x_{i_j},x_{i_k})
\end{equation}
as $(i_1,\ldots,i_{|B|})$ ranges over all orderings of $B$.
The possible products are $\sum_{i\in B}x_i+\Omega(B)z$.
A block is \emph{full} if $\Omega(B)=\F$.
Its nonorthogonality graph joins occurrences with nonzero pairing.
An oriented triple $(u,v,w)$ is \emph{balanced} if
\[
\omega(u,v)=\omega(v,w)=\omega(w,u)\ne0.
\]
This definition is invariant under reordering the three vertices.
Occupancy always counts occurrences, including both nonzero values
on a one-dimensional subspace.

\section{A relative moment criterion}
\label{sec:relative-moments}

The subgroup used in a moment reduction need not be central.  This
section establishes a criterion for an arbitrary abelian subgroup
containing the derived subgroup.  All sequences and subsequences are
indexed: equal group elements occurring at different positions remain
different occurrences.

Throughout this section, let $G$ be a finite group of exponent three and
nilpotency class at most two, with
$G'=\langle z\rangle$ of order three.  We use its class-two Lie
coordinates over $\mathbb F_3$, in which multiplication is
\begin{equation}
 x*y=x+y+\frac12\omega(x,y)z.
 \label{eq:relative-lie-product}
\end{equation}
Here the underlying vector space is denoted by $\mathfrak g$, the
scalar $1/2$ means $2\in\mathbb F_3$, and $\omega$ is an alternating
bilinear form satisfying $\omega(z,x)=0$ for every $x$.
For completeness, with the group commutator convention
$[x,y]=x^{-1}y^{-1}xy$, the additive operation is
$x+y=xy[x,y]^{-1/2}$ and scalar multiplication is given by group
powers.  The class-two commutator identities verify that this is a
vector space over $\mathbb F_3$ and that
\eqref{eq:relative-lie-product} recovers the group multiplication,
where $[x,y]=z^{\omega(x,y)}$.

Let $A\le G$ be abelian and contain $G'$.  It is a vector subspace in
these coordinates, and it is normal because $[G,A]\subseteq G'\le A$.
Write
\[
 \dim_{\mathbb F_3} A=b,
 \qquad
 \dim_{\mathbb F_3}(G/A)=a.
\]
The restriction of $\omega$ to $A\times A$ is zero, but its
restriction to $A\times\mathfrak g$ is not assumed to be zero.

\subsection{Coefficient functions and ordering corrections}

We first record the coefficient-degree fact used below.  The degree of
a function on $\mathbb F_3^b$ means the total degree of its unique
polynomial representative having degree at most two in every
coordinate.

\begin{lemma}[Degree of an augmentation coefficient function]
\label{lem:relative-coefficient-degree}
Let $V=\mathbb F_3^b$, let $\mathbb F_3[V]$ have basis
$\{X^v:v\in V\}$ with multiplication $X^uX^v=X^{u+v}$, and let $I$
be its augmentation ideal.  If $F\in I^m$ and
$f(v)=[X^v]F$, then $f$ has degree at most $2b-m$.
If $m>2b$, both $F$ and $f$ are zero.

In particular, for an indexed sequence $C$ in $V$, the function
\[
 f_C(v)=[X^v]\prod_{c\in C}(1-X^c)
\]
has degree at most $2b-|C|$.  If $C$ is zero-sum-free, then
$f_C(0)=1$.
\end{lemma}

\begin{proof}
Choose a basis $e_1,\ldots,e_b$ of $V$ and put $Y_j=1-X^{e_j}$.
There is an isomorphism
\[
 \mathbb F_3[V]\cong
 \mathbb F_3[Y_1,\ldots,Y_b]/(Y_1^3,\ldots,Y_b^3),
 \qquad I=(Y_1,\ldots,Y_b).
\]
In one coordinate, the coefficient of $X^{te_j}$ in
$Y_j^k$, for $t\in\mathbb F_3$ and $k=0,1,2$, is respectively
\[
 p_0(t)=1-t^2,
 \qquad p_1(t)=1+t,
 \qquad p_2(t)=1.
\]
Thus the coefficient function of $Y_1^{k_1}\cdots Y_b^{k_b}$ has
degree at most $\sum_j(2-k_j)=2b-\sum_jk_j$.
The ideal $I^m$ is spanned by the monomials with
$\sum_jk_j\ge m$ and $0\le k_j\le2$, proving the asserted bound and
the vanishing for $m>2b$.  Each factor $1-X^c$ belongs to $I$.
Finally, if $C$ is zero-sum-free, the empty subset is the only subset
contributing to the coefficient of $X^0$, so that this coefficient is
one.
\end{proof}

For a nonempty indexed block $B$ of vectors $x_i\in\mathfrak g$,
define its set of ordering corrections by
\begin{equation}
 \Omega(B)=
 \left\{
 \frac12\sum_{p<q}\omega(x_{\sigma(p)},x_{\sigma(q)}):
 \sigma\text{ is an ordering of }B
 \right\}\subseteq\mathbb F_3.
 \label{eq:relative-omega}
\end{equation}
Its possible products are exactly
$\sum_{i\in B}x_i+t z$, for $t\in\Omega(B)$.
The nonorthogonality graph of $B$ has the occurrences in $B$ as its
vertices and an edge $\{i,j\}$ whenever $\omega(x_i,x_j)\ne0$.
A triangle on occurrences $i,j,k$ is called \emph{balanced} if
\begin{equation}
 \omega(x_i,x_j)=\omega(x_j,x_k)=\omega(x_k,x_i)\ne0.
 \label{eq:relative-balanced}
\end{equation}
This property does not depend on the chosen cyclic orientation.

\begin{lemma}[The nonfull correction sets]
\label{lem:relative-order-types}
For a nonempty block $B$, precisely the following nonfull cases occur:
\begin{enumerate}
 \item the graph has no edges, in which case $\Omega(B)=\{0\}$;
 \item the graph consists of one edge and isolated vertices, in which
 case $\Omega(B)=\{-1,1\}$;
 \item the graph consists of one balanced triangle and isolated
 vertices, in which case $\Omega(B)=\{-1,1\}$.
\end{enumerate}
Every other graph, with its specified nonzero pairings, has
$\Omega(B)=\mathbb F_3$.
\end{lemma}

\begin{proof}
If a subblock has all three corrections, so does the whole block:
place the subblock consecutively in a fixed ordering of the remaining
occurrences.  The contribution from pairs crossing its boundary is
independent of the internal ordering of the subblock.

Two disjoint edges already give all three corrections.  Indeed, place
the endpoints of each edge consecutively, and independently swap the
two adjacent pairs.  The four resulting corrections have the form
$c+\{0,u\}+\{0,v\}$ with $u,v\in\mathbb F_3^*$, and this set is
$\mathbb F_3$.  Hence a nonfull graph has matching number at most one.
A graph with this property, after deleting isolated vertices, is a
star or a triangle: if two edges share a vertex and an edge avoids that
vertex, the three edges form a triangle, and every further edge must
belong to that triangle.

A star with at least two edges is full.  For its center and two
leaves, the two nonzero pairings can independently receive either
sign according to the positions of the leaves before or after the
center.  Their signed sums contain all of $\mathbb F_3$.
For a triangle, put
$u=\omega(x_i,x_j)$, $v=\omega(x_j,x_k)$, and
$w=\omega(x_k,x_i)$.  Its corrections are
\[
 \left\{
 \pm\tfrac12(u+v-w),\quad
 \pm\tfrac12(u-v-w),\quad
 \pm\tfrac12(-u+v-w)
 \right\}.
\]
Since $u,v,w\in\{1,-1\}$, this set is $\{1,-1\}$ when
$u=v=w$, and is $\mathbb F_3$ otherwise.  A single edge has corrections
$\{\pm\frac12\omega(x_i,x_j)\}=\{1,-1\}$, whereas an edgeless
block has correction zero.  Isolated vertices do not change any of
these correction sets.
\end{proof}

We also use the elementary Boolean identity
\begin{equation}
 \sum_{B\subseteq[n]}(-1)^{|B|}P(\mathbf1_B)=0
 \quad\text{whenever}\quad \deg P<n.
 \label{eq:relative-boolean}
\end{equation}
It holds over $\mathbb F_3$: every monomial of degree less than $n$
omits some variable, and summing over that Boolean variable cancels
its contribution.

\subsection{The criterion}

Suppose that $S$ is a product-one-free sequence in $G$ of length
$2(a+b)+1$.  Let $C$ be the subsequence consisting of its occurrences
in $A$, and suppose
\[
 |C|=2b-d,
 \qquad d\ge2.
\]
Index the remaining occurrences by $[n]$, where
\[
 n=2a+d+1,
 \qquad D=d-2.
\]
Choose a linear complement to $A$ in $\mathfrak g$ and write each
outside occurrence as $x_i=v_i+w_i$, where $v_i$ belongs to the
complement, identified with $G/A$, and $w_i\in A$.
A block $B\subseteq[n]$ is called \emph{quotient-zero} if
$\sum_{i\in B}v_i=0$.

Let $\mathcal F$ consist of the nonempty quotient-zero blocks with
$\Omega(B)=\mathbb F_3$, and let $\mathcal T$ consist of the nonempty
quotient-zero blocks whose graph is a balanced triangle together with
isolated vertices.  For each block $B$, define the moment row
\begin{equation}
 r_D(B)=
 \left(\prod_{i\in H}\mathbf1_B(i)\right)_
 {H\subseteq[n],\ |H|\le D},
 \label{eq:relative-moment-row}
\end{equation}
including the coordinate indexed by $H=\varnothing$, which equals
one.  All vector spaces and row spans below are over $\mathbb F_3$.
Set
\begin{equation}
 M_D=\operatorname{span}\{r_D(B):B\in\mathcal F\},
 \qquad
 W_D=\sum_{B\in\mathcal T}(-1)^{|B|}r_D(B).
 \label{eq:relative-MW}
\end{equation}

\begin{theorem}[Relative moment criterion]
\label{thm:relative}
Under the preceding hypotheses,
\[
 W_D\notin M_D.
\]
More precisely, there is a linear functional on the moment-row space
which vanishes on $M_D$ and takes the value one on $W_D$.
Consequently, verifying $W_D\in M_D$ for a class of outside
configurations excludes every product-one-free sequence of the stated
length and subgroup occupancy having a configuration in that class.
No centrality assumption on $A$ is required.
\end{theorem}

\begin{proof}
Since $A$ is abelian, its group multiplication is its additive
operation in Lie coordinates.  The subsequence $C$ is therefore
zero-sum-free in $A$.  By
Lemma~\ref{lem:relative-coefficient-degree}, its coefficient function
\[
 f(w)=[X^w]\prod_{c\in C}(1-X^c)
\]
has degree at most $d$ and satisfies $f(0)=1$.
Put $g(w)=f(-w)$.  Choose coordinates $w=s z+y$ on $A$, with $y$ in
a complement to $\mathbb F_3z$, and write the reduced polynomial as
\[
 g(s z+y)=g_0(y)+s g_1(y)+s^2g_2(y).
\]
Define $h(s z+y)=g_2(y)$.  Then $h$ is independent of $s$ and has
degree at most $d-2=D$.  Equivalently, $h(w)$ is the coefficient of
$t^2$ in $g(w+t z)$.

For a quotient-zero block $B$, put
$L_B=\sum_{i\in B}w_i\in A$.  Its products are
$L_B+t z$, with $t\in\Omega(B)$.  Product-one freeness implies
\begin{equation}
 g(L_B+t z)=0
 \qquad(B\ne\varnothing,\quad t\in\Omega(B)).
 \label{eq:relative-POF-vanishing}
\end{equation}
Indeed, a nonzero value of $f(-L_B-t z)$ would give a subset of $C$
with additive sum $-L_B-t z$.  Concatenating it with the corresponding
ordering of $B$ would produce a nonempty product-one subsequence of
$S$.  The concatenation argument only uses that the product of $B$
and the selected product from $C$ lie in $A$; individual outside
occurrences are not required to commute with $A$.

The polynomial $t\mapsto g(L_B+t z)$ has degree at most two.
If $B\in\mathcal F$, it vanishes at all three elements of
$\mathbb F_3$, so
\begin{equation}
 g(L_B)=h(L_B)=0.
 \label{eq:relative-full-vanishing}
\end{equation}
If $B$ is noncommuting and nonfull, its correction set is
$\{-1,1\}$ by Lemma~\ref{lem:relative-order-types}.  Substituting
these two values in \eqref{eq:relative-POF-vanishing} gives
\begin{equation}
 g(L_B)=-h(L_B).
 \label{eq:relative-two-values}
\end{equation}
If $B$ is nonempty and commuting, its correction set is $\{0\}$ and
$g(L_B)=0$.

Write $v_i=(v_{i1},\ldots,v_{ia})$ in quotient coordinates and put
\[
 Q(X_1,\ldots,X_n)=
 \prod_{j=1}^{a}\left(1-
       \left(\sum_{i=1}^{n}v_{ij}X_i\right)^2\right).
\]
On the Boolean cube, $Q(\mathbf1_B)$ is one for quotient-zero
blocks and zero for all other blocks.  Apply
\eqref{eq:relative-boolean} to
\[
 P(X)=Q(X)g\!\left(\sum_{i=1}^{n}w_iX_i\right),
 \qquad \deg P\le 2a+d=n-1.
\]
The empty block contributes $g(0)=1$.  The full and nonempty
commuting blocks contribute zero.  If $\mathcal N$ denotes the
nonempty quotient-zero blocks that are noncommuting and nonfull,
equation~\eqref{eq:relative-two-values} therefore gives
\begin{equation}
 \sum_{B\in\mathcal N}(-1)^{|B|}h(L_B)=1.
 \label{eq:relative-nonfull-sum}
\end{equation}

Boolean multilinearization of the degree-at-most-$D$ polynomial
$h(\sum_i w_iX_i)$ gives a linear functional $\lambda$ on the moment
rows satisfying
\[
 \lambda(r_D(B))=h(L_B)
 \qquad\text{for every }B\subseteq[n].
\]
Here $L_B=\sum_{i\in B}w_i$ may also be used for blocks that are not
quotient-zero.  Replacing positive powers of $X_i$ by $X_i$ on the
Boolean cube does not increase the total degree.  Equation
\eqref{eq:relative-full-vanishing} shows that
$\lambda(M_D)=0$.

It remains to remove the single-edge blocks from
\eqref{eq:relative-nonfull-sum}.  Let $\mathcal E$ be the family of
nonempty quotient-zero blocks whose graph consists of one edge and
isolated vertices.  By Lemma~\ref{lem:relative-order-types},
$\mathcal N$ is the disjoint union of $\mathcal E$ and
$\mathcal T$.
Consider the unweighted edge-count polynomial
\[
 e(X)=\sum_{\substack{i<j\\\omega(x_i,x_j)\ne0}}X_iX_j.
\]
For every $H\subseteq[n]$ with $|H|\le D$, the polynomial
\[
 Q(X)e(X)\prod_{i\in H}X_i
\]
has degree at most $2a+2+D=n-1$.  Combining its Boolean identities
\eqref{eq:relative-boolean} coordinate by coordinate yields
\[
 0=\sum_{B\subseteq[n]}(-1)^{|B|}
       Q(\mathbf1_B)e(\mathbf1_B)r_D(B).
\]
The edge count is zero on commuting blocks, one on blocks in
$\mathcal E$, and three, hence zero in $\mathbb F_3$, on blocks in
$\mathcal T$.  All remaining quotient-zero blocks are full.
Consequently,
\begin{equation}
 \sum_{B\in\mathcal E}(-1)^{|B|}r_D(B)
 =-\sum_{B\in\mathcal F}(-1)^{|B|}
       e(\mathbf1_B)r_D(B)
 \ \in M_D.
 \label{eq:relative-single-edge-removal}
\end{equation}
Applying $\lambda$ to
\eqref{eq:relative-nonfull-sum} and using
\eqref{eq:relative-single-edge-removal} gives $\lambda(W_D)=1$.
Since $\lambda(M_D)=0$, this proves the theorem.
\end{proof}

Theorem~\ref{thm:relative} is a necessary condition on a
product-one-free sequence, not a universal assertion that the moment
membership $W_D\in M_D$ holds.  Each geometric application must
establish that membership for its specified outside configurations.
The single-edge cancellation in
\eqref{eq:relative-single-edge-removal} is unweighted; it does not,
without a further argument, cancel an arbitrary weighted sum of
single-edge moment rows.

\section{Initial subspace bounds and the active-line improvement}
\label{sec:line-caps}

Throughout this section, let
\[
 G=E_2\times C_3^r,\qquad r\geq1,
\]
and suppose that $S$ is a product-one-free sequence of length $2r+11$
with exactly $2r-1$ terms in $Z(G)$. Write
$P=G/Z(G)\cong\F^4$ for its active quotient, with its nondegenerate
alternating form $\omega$. For $U\leq P$, let $N_U$ be its full
inverse image and let $n_S(U)$ count the noncentral occurrences of $S$
whose active projections lie in $U$. Thus
\begin{equation}\label{eq:lc-count}
 |S\cap N_U|=(2r-1)+n_S(U).
\end{equation}
All occurrences, including opposite active vectors, are counted with
multiplicity. In particular, the bound on a line concerns both of its
nonzero vectors together.

\subsection{Two subgroup bounds}

We first record the elementary subgroup arguments needed for the initial
caps. The only previously published nonabelian Davenport value used in
this section is
\begin{equation}\label{eq:lc-published}
 \dc(H_{27}\times C_3^s)=2s+6\qquad(s\geq0),
\end{equation}
proved in~\cite{VolkmannH27}. No value of $\dc(E_2\times C_3^s)$
is assumed.

\begin{lemma}[Ordinary subgroup saturation]\label{lem:lc-saturation}
Let $N\trianglelefteq K$ be finite groups and let $T$ be product-one-free
with $|T|\geq\dc(N)+\dc(K/N)+1$. Then
$|T\cap N|\leq\dc(N)-1$.
\end{lemma}
\begin{proof}
More than $\dc(N)$ terms in $N$ are impossible. If $C=T\cap N$ has
length $\dc(N)$, its subsequence products, including the empty product,
cover $N$. Indeed, append $h^{-1}$ for any $h\in N$. A product-one
subsequence must use the appended occurrence, and cyclically rotating a
product-one ordering places that occurrence last. The other selected
terms then have product $h$.

The remaining at least $\dc(K/N)+1$ terms admit a nonempty subsequence
with a product in $N$. That product can be cancelled by a disjoint
subsequence of $C$, a contradiction.
\end{proof}

\begin{lemma}[Abelian defect-one exclusion]\label{lem:lc-abelian}
Let $K$ have exponent three and nilpotency class at most two. Let
$A\leq K$ be abelian with $K'\leq A$, $|A|=3^b$ and
$|K/A|=3^a$, where $b\geq1$. If $T$ is product-one-free of length
$2(a+b)+1$, then
\[
 |T\cap A|\leq2b-2.
\]
\end{lemma}
\begin{proof}
Suppose that $T\cap A$ contains a subsequence $C$ of length $2b-1$.
In the additive group algebra of $A$, put
\[
 f(w)=[X^w]\prod_{c\in C}(1-X^c),\qquad g(w)=f(-w).
\]
Lemma~\ref{lem:relative-coefficient-degree} makes $g$ affine, and
$g(0)=1$ because $C$ is zero-sum-free. Choose a linear complement to
$A$ in the class-two Lie coordinates of $K$ and write the remaining
$2a+2$ terms as $v_i+w_i$, with $v_i\in\F^a$ and $w_i\in A$.
The complement need not be a Lie subalgebra.

Consider a nonempty selection $B$ with $\sum_{i\in B}v_i=0$.
Every ordered product of its terms belongs to $A$. If its Lie coordinate
is $Q$, then $g(Q)=0$: otherwise a subsequence of $C$ would have product
$-Q$ and cancel it. An ordering and its reversal have Lie coordinates
$L+c$ and $L-c$, where $L=\sum_{i\in B}w_i$ and $c\in K'\leq A$.
Affinity, and the invertibility of two, therefore give $g(L)=0$.
Consequently the polynomial
\[
 \prod_{j=1}^{a}\left(1-
     \left(\sum_{i=1}^{2a+2}v_{ij}X_i\right)^2\right)
 g\!\left(\sum_{i=1}^{2a+2}w_iX_i\right)
\]
is the indicator of the empty selection on the Boolean cube. Its degree
is at most $2a+1$, whereas the unique multilinear representation of
that indicator, $\prod_i(1-X_i)$, has degree $2a+2$. Boolean
multilinearization cannot increase degree, giving a contradiction.
\end{proof}

The identity $\dc(C_3^k)=2k$ can be used here without an additional
nonabelian input: the presentation
$\F[C_3^k]\cong\F[T_1,\ldots,T_k]/(T_1^3,\ldots,T_k^3)$
gives the upper bound, since an augmentation product of length $2k+1$
vanishes while a zero-sum-free sequence would give identity coefficient
one. Two copies of each basis vector give the lower bound.

\begin{proposition}[Initial caps]\label{prop:lc-initial}
Under the standing assumptions,
\[
\begin{array}{c|c}
 U& n_S(U)\text{ is at most}\\\hline
 \text{line}&3\\
 \text{isotropic plane}&5\\
 \text{nondegenerate plane}&6\\
 \text{hyperplane}&8.
\end{array}
\]
\end{proposition}
\begin{proof}
For an isotropic subspace of dimension $k\in\{1,2\}$, the subgroup
$N_U$ is elementary abelian of dimension $b=r+1+k$, and
$G/N_U\cong C_3^{4-k}$. Since $2(b+4-k)+1=2r+11$,
Lemma~\ref{lem:lc-abelian} gives
$|S\cap N_U|\leq2r+2k$. Subtracting the $2r-1$ central terms
proves the first two bounds.

For a nondegenerate plane, $N_U\cong H_{27}\times C_3^r$ and
$G/N_U\cong C_3^2$. Equation~\eqref{eq:lc-published} and
Lemma~\ref{lem:lc-saturation} give
$|S\cap N_U|\leq2r+5$, hence $n_S(U)\leq6$.
For a hyperplane, the restricted alternating form has rank two and a
one-dimensional radical, so
$N_U\cong H_{27}\times C_3^{r+1}$ and $G/N_U\cong C_3$.
The same argument gives $|S\cap N_U|\leq2r+7$, hence
$n_S(U)\leq8$. In both applications the required length is exactly
$\dc(N_U)+\dc(G/N_U)+1=2r+11$.
\end{proof}

\subsection{A quadratic obstruction above an abelian subgroup}

The following elementary argument will improve the line cap independently
of the signs of the three terms on a saturated line.

\begin{lemma}[Quadratic-fibre obstruction]\label{lem:lc-quadratic}
Let $K$ have exponent three and nilpotency class two, with
$K'=\langle z\rangle$ of order three. Let $A\geq K'$ be abelian,
$|A|=3^b$, and $|K/A|=3^a$. Suppose that $T$ is product-one-free,
$|T|=2(a+b)+1$, and $|T\cap A|=2b-2$.
Then no nonempty block of $T\setminus A$ with quotient sum zero in
$K/A$ has ordering-correction set $\F$.
\end{lemma}
\begin{proof}
Let $C=T\cap A$ and
\[
 g(w)=[X^{-w}]\prod_{c\in C}(1-X^c).
\]
By Lemma~\ref{lem:relative-coefficient-degree}, $g$ has degree at most
two and $g(0)=1$. For any nonempty quotient-zero outside block $B$,
write $L_B\in A$ for its additive Lie sum and $\Omega(B)$ for its
scalar ordering-correction set. Product-one freeness implies
\begin{equation}\label{eq:lc-fibre-zeros}
 g(L_B+qz)=0\qquad(q\in\Omega(B)).
\end{equation}
This uses only that the terms of $C$ commute with one another; $A$ need
not be central.

The coefficient of $t^2$ in $g(w+tz)$ is independent of $w$. If one
outside quotient-zero block is full,\footnote{Here and below, a block is
\emph{full} when its scalar ordering-correction set is all of $\F$.}
\eqref{eq:lc-fibre-zeros} forces that coefficient to vanish. Thus $g$
is affine on every line parallel to $z$.
For any other quotient-zero outside block, commuting terms have correction
zero. Noncommuting terms have at least two correction values by an adjacent
exchange, and their correction set is stable under negation by reversal;
over $\F$ it therefore contains $1$ and $-1$. In either case,
\eqref{eq:lc-fibre-zeros} and affinity give $g(L_B)=0$.

There are $2a+3$ outside terms. Writing them as $v_i+w_i$ in a linear
splitting over $A$, the polynomial
\[
 \prod_{j=1}^{a}\left(1-
       \left(\sum_i v_{ij}X_i\right)^2\right)
       g\!\left(\sum_i w_iX_i\right)
\]
is again the empty-selection indicator. Its degree is at most $2a+2$,
strictly smaller than the number of Boolean variables, a contradiction.
\end{proof}

\subsection{The finite nine-term line lemma}

\begin{lemma}[Nine terms outside a line]\label{lem:lc-nine}
Let $L=\langle x\rangle$ be a line in a nondegenerate symplectic
space $P\cong\F^4$. Let $R$ be a sequence of nine vectors outside $L$
satisfying the following bounds, all counted with multiplicity:
\begin{enumerate}
\item every line other than $L$ contains at most three terms;
\item every isotropic plane through $L$ contains at most two terms,
and every nondegenerate plane through $L$ at most three;
\item every plane not containing $L$ contains at most five terms;
\item every hyperplane through $L$ contains at most five terms.
\end{enumerate}
Then $R$ has a nonempty block $B$ with $\sum B\in L$ and
$\Omega(B)=\F$.
\end{lemma}
\begin{proof}[Finite verification and its coverage]
First observe an equivalent witness condition: at least one of
$R\cup\{x\}$ and $R\cup\{-x\}$ contains an active zero-sum block
with two disjoint nonorthogonal pairs.
For a zero-sum block, its nonorthogonality graph has no vertex of degree
one, since pairing that vertex with the total sum would give a nonzero
value. A graph with no two disjoint edges is a star or a triangle, with
possible isolates. Consequently a zero-sum block without the specified
matching is either commuting or a balanced triangle with isolates, and
its correction set is respectively $\{0\}$ or $\{1,-1\}$.
Conversely, two disjoint nonorthogonal pairs give all three corrections
by their independent adjacent exchanges.

If $\sum B=\pm x$, append the opposite generator. This preserves the
correction set: in a zero-active-sum block, cyclic rotation changes no
ordering correction, and putting the appended term last contributes
$\frac12\omega(\sum B,-\sum B)=0$.
The same deletion argument proves the reverse implication. A block
with sum zero needs no appended term.

Here is the complete finite search. At least one term $y$ of $R$ pairs
nontrivially with $x$, since otherwise all nine terms lie in the
hyperplane $x^\perp$, violating (4). Choose the sign of $x$ so that
$\omega(x,y)=1$, and use a symplectic transformation to fix this pair.
For the primary implementation the coordinates are $(a_1,a_2,b_1,b_2)$,
encoded by $a_1+3a_2+9b_1+27b_2$; $x$ has code $1$ and the distinguished
outside term $y$ code $9$. The other eight terms are enumerated as a
nondecreasing multiset of codes $3,\ldots,80$, with repetitions allowed.

For each prefix, the search rejects a violated cap or an already present
witness. Both properties are inherited by extensions. It also considers
both generators of $L$ and each present term pairing nontrivially with
that generator, normalizes the chosen pair, and tests all 24 completions
to a symplectic basis. After removing one distinguished term of code $9$,
the smallest sorted remaining list is retained.
This prefix canonicalization is lossless: choose a lexicographically
smallest normalized representative of a hypothetical full counterexample.
A smaller normalized image of one of its prefixes would give a smaller
full representative under the same transformation, since sorting the
additional future terms can only decrease the initial sorted positions.
Thus that representative has no rejected prefix.

The standalone ancillary program
\texttt{verify\_line9\_primary.cpp} implements this search using subset
sums and matchings. It constructs all 40 lines and 130 planes, checks the
four isotropic and nine nondegenerate planes through $L$ and the 13
hyperplanes through $L$, and verifies the 51,840 symplectic maps used for
normalization. A separate standalone program,
\texttt{verify\_line9\_independent.cpp}, uses the coordinate pairings
$(0,3),(1,2)$, a different anchor, and a direct construction of only the
1,296 maps stabilizing $L$. It verifies form preservation on all $81^2$
vector pairs. Instead of matchings, it computes exact correction sets
by the first-element recursion
\[
 \Omega(B)=\bigcup_{i\in B}
 \left(\Omega(B\setminus\{i\})+
 \tfrac12\omega\left(v_i,\sum_{j\in B\setminus\{i\}}v_j\right)\right),
 \qquad \Omega(\varnothing)=\{0\},
\]
including the two possible appended generators. Its recursion is also
checked against all permutations in 768 specified subset-and-sign cases;
this is an implementation check, not a substitute for the exhaustive
search.

Both independent searches terminate completely, with the same accepted
prefix counts at lengths $1$ through $9$:
\begin{equation}\label{eq:lc-counts}
 1,\ 11,\ 117,\ 1388,\ 14537,\ 89660,\ 176940,\ 4643,\ 0.
\end{equation}
There is therefore no full counterexample. The primary completion status
is \file{EXHAUSTED_ALL_HAVE_SMALL_FULL_WITNESS}; the independent
status is \file{EXHAUSTED_NO_COUNTEREXAMPLE}. A timeout is a separate
non-successful status. The proof uses neither a quotient $2+2$ exclusion
nor any cap on hyperplanes not containing $L$.
\end{proof}

\subsection{The required active-line cap}

\begin{proposition}[Improved initial caps]\label{prop:lc-final}
Every hypothetical sequence $S$ under the standing assumptions satisfies
\[
\boxed{
 n_S(L)\leq2,\qquad
 n_S(U_{\mathrm{iso}})\leq5,\qquad
 n_S(U_{\mathrm{nd}})\leq6,\qquad
 n_S(H)\leq8.}
\]
\end{proposition}
\begin{proof}
Only the first bound remains. Suppose that a line $L$ contains three
active occurrences. The initial line bound shows that these are all of
its active occurrences. Its full inverse image $A=N_L$ is abelian of
dimension $b=r+2$, and $G/A$ has dimension $a=3$. Its intersection with
$S$ has length
\[
 |S\cap A|=(2r-1)+3=2r+2=2b-2.
\]
The nine remaining terms have active projections $R$ outside $L$.

Subtracting the three line occurrences from the initial caps gives two
outside terms in each isotropic plane through $L$, three in each
nondegenerate plane through $L$, and five in each hyperplane through $L$.
Any other line lies with $L$ in a plane and hence contains at most three
outside terms. If a plane $U$ does not contain $L$, its span $U+L$ is a
hyperplane, so $U$ contains at most five outside terms. These are exactly
the hypotheses of Lemma~\ref{lem:lc-nine}.

That lemma provides a nonempty outside block with active sum in $L$ and
full correction set. Its quotient sum in $G/A$ is zero, contradicting
Lemma~\ref{lem:lc-quadratic}. This argument is unaffected by the signs
of the three line projections or by any central coordinates of their
lifts. Thus both nonzero vectors on $L$ together have multiplicity at
most two, as required.
\end{proof}

\section{Isotropic plane bounds in the twelve-term layer}
\label{sec:isotropic-caps}

We record two finite geometric statements and their algebraic
consequences.  All multiplicities in this section count occurrences,
including occurrences with opposite signs on the same projective line.
Let $(P,\omega)$ be a nondegenerate four-dimensional symplectic vector
space over $\mathbb F_3$, and let $U\leq P$ be an isotropic plane.
For an indexed block $B$ of vectors $v_i\in P$, write
\[
 \sigma(B)=\sum_{i\in B}v_i,
 \qquad
 \Omega(B)=\left\{
   \frac12\sum_{a<b}\omega(v_{\pi(a)},v_{\pi(b)}):
   \pi\text{ is an ordering of }B
 \right\}\subseteq\mathbb F_3.
\]
Here $1/2=2$ in $\mathbb F_3$.  The empty block has
$\Omega(\varnothing)=\{0\}$.  A block is called \emph{full} if
$\Omega(B)=\mathbb F_3$; the quotient-zero condition in this section
is $\sigma(B)\in U$.  In particular, quotient-zero does not require
$\sigma(B)=0$ in $P$.

\begin{lemma}[Seven outside vectors]\label{lem:isotropic-seven}
Let $R=(v_1,\ldots,v_7)$ be a sequence in $P\setminus U$ satisfying
\[
\begin{array}{c|c}
\text{subspace}&\text{maximum number of occurrences of }R\\ \hline
\text{one-dimensional subspace}&2\\
\text{isotropic plane}&5\\
\text{nondegenerate plane}&6\\
\text{hyperplane containing }U&3.
\end{array}
\]
There is a nonempty block $B\subseteq[7]$ such that
$\sigma(B)\in U$ and $\Omega(B)=\mathbb F_3$.
\end{lemma}

\begin{lemma}[Eight outside vectors]\label{lem:isotropic-eight}
Let $R=(v_1,\ldots,v_8)$ be a sequence in $P\setminus U$ satisfying
\[
\begin{array}{c|c}
\text{subspace}&\text{maximum number of occurrences of }R\\ \hline
\text{one-dimensional subspace}&2\\
\text{isotropic plane}&4\\
\text{nondegenerate plane}&6\\
\text{hyperplane containing }U&4.
\end{array}
\]
For $B\subseteq[8]$, let
$r(B)=(1,\mathbf1_B)\in\mathbb F_3^9$, and put
\[
\begin{split}
 \mathcal F_U&=\{\varnothing\ne B\subseteq[8]:
       \sigma(B)\in U,\ \Omega(B)=\mathbb F_3\},\\
 \mathcal N_U&=\{\varnothing\ne B\subseteq[8]:
       \sigma(B)\in U,\ \Omega(B)=\{-1,1\}\},\\
 M_U&=\operatorname{span}_{\mathbb F_3}
                  \{r(B):B\in\mathcal F_U\},\qquad
 W_U=\sum_{B\in\mathcal N_U}(-1)^{|B|}r(B).
\end{split}
\]
Then $W_U\in M_U$.
\end{lemma}

\subsection{Complete finite verification of the two lemmas}
\label{subsec:isotropic-enumeration}

We give the enumeration conventions and the proofs of the pruning
rules, so that the finite assertions are reproducible from the supplied
integer-arithmetic programs.  Neither lemma is inferred from random
sampling.

In the first implementation, use coordinates
$(a_1,a_2,b_1,b_2)$, encoded by
$a_1+3a_2+9b_1+27b_2$, and the pairing
\[
 \omega(v,w)=a_1b'_1-b_1a'_1+a_2b'_2-b_2a'_2.
\]
The plane $U$ consists of codes $0,\ldots,8$; there are $72$ outside
vectors.  The geometry contains $40$ projective points, $130$ planes,
and four hyperplanes through $U$.  The stabilizer of $U$ in
$\operatorname{Sp}_4(3)$ has $1296$ elements.  It acts transitively
on the outside vectors, with exactly $18$ maps taking any specified
outside vector to the anchor of code $9$.

The programs construct these maps from all $51840$ symplectic maps,
retain precisely those preserving $U$, and verify the stated
cardinalities.  Thus one outside occurrence can be fixed at the anchor.
The other occurrences are listed in nondecreasing order.  At every
prefix, all images obtained by mapping a present occurrence to the
anchor are considered; one anchor occurrence is removed, and the
remaining images are sorted.  A prefix is rejected if a smaller
normalized tail occurs.

This prefix normalization loses no counterexample.  Choose a
lexicographically least normalized representative of a putative full
counterexample.  If one of its sorted prefixes admitted a smaller
normalized image, apply the same map to the full sequence and remove
the same anchor occurrence.  Adding the remaining images before sorting
cannot increase the initial order statistics of the smaller prefix.
The resulting full normalized tail would be smaller, a contradiction.
Cap violations are hereditary under extension and can therefore also
be rejected at the prefix stage.

The first implementation computes every ordering-correction set exactly
by the recurrence
\begin{equation}\label{eq:outside-correction-recurrence}
 \Omega(B)=\bigcup_{j\in B}\left(
  \Omega(B\setminus\{j\})+
  \frac12\omega\bigl(\sigma(B\setminus\{j\}),v_j\bigr)
 \right).
\end{equation}
This follows by specifying the last term in an ordering of $B$.
It uses only subsets of the current prefix.  In particular, the test
is not restricted to blocks with a pair of disjoint nonorthogonal
edges.

For Lemma~\ref{lem:isotropic-seven}, the branch is closed as soon as
a nonempty full block has sum in $U$.  Such a block persists in every
extension.  The accepted prefix counts at lengths one through seven
are
\[
 1,\quad 8,\quad 65,\quad 617,\quad 2721,\quad 189,\quad 0.
\]
Thus no complete counterexample survives.  The supplied source and log
are \file{verify_plane7_primary.cpp} and
\file{verify_plane7_primary.log}.

For Lemma~\ref{lem:isotropic-eight}, each full quotient-zero row is
inserted into a row basis by exact Gaussian elimination over
$\mathbb F_3$.  The target includes \emph{every} quotient-zero block
with a two-element correction set, with sign $(-1)^{|B|}$.  In
particular, single-edge blocks are included in the computation.
At a complete leaf the target is reduced against the full-row basis;
all deficient-rank leaves are checked in this way.

One additional hereditary closure is used.  Suppose a prefix of length
$n$ has quotient rank two.  Its augmented quotient-zero kernel is
\[
 K_n=\left\{(c,x_1,\ldots,x_n):
           \sum_{i=1}^n x_i(v_i+U)=0\right\},
 \qquad \dim K_n=n-1.
\]
If its full quotient-zero rows span $K_n$, and its full blocks have
sums in all eight nonzero classes of $P/U$, then any further outside
vector $v$ extends one of these full blocks to a quotient-zero full
block.  Its row has a fresh occurrence coordinate equal to one, while
all old rows have zero there.  Consequently the full-row rank increases
by one and again reaches the dimension of the augmented kernel.
Fullness persists because placing $v$ last translates the entire old
correction set by a fixed scalar.  The original full blocks remain
available, so the same argument applies at every subsequent extension.
Such a prefix is therefore closed without visiting its descendants.
Every admissible length-eight sequence has quotient rank two: otherwise
all eight terms would lie in one hyperplane through $U$, contrary to its
bound four.  At length eight, full-row rank seven alone therefore proves
the asserted target membership.

The complete ascending enumeration has accepted prefix counts
\[
 1,\ 8,\ 73,\ 1028,\ 14059,\ 172604,\ 1505040,\ 3325333.
\]
It closes $37831$ prefixes of length six and $1287248$ prefixes of
length seven by the hereditary coverage rule.  All $3325333$ visited
complete leaves satisfy $W_U\in M_U$.  Of these, $320$ have full-row
rank below seven and pass explicit target reduction.  The source and
log are \file{verify_plane8_primary.cpp} and
\file{verify_plane8_primary.log}.  Both sources are standalone; their seven- and
eight-term searches are separate.

\subsection{A separately implemented exhaustive verification}
\label{subsec:isotropic-independent}

The standalone programs
\file{verify_plane7_independent.cpp} and
\file{verify_plane8_independent.cpp}
provide complete second enumerations.  They use coordinate pairing
$(0,3),(1,2)$, anchor $27$, descending normalized tails, and a direct
construction of the $U$-stabilizer from basis images.  They verify all
$1296$ distinct maps, pairing preservation, and the $18$ normalizers
per outside vector.  The eight-term program also uses the opposite
Gaussian pivot order.

These programs recognize correction sets by their nonorthogonality
graph rather than by \eqref{eq:outside-correction-recurrence}.
The classification used is exact: an empty graph gives $\{0\}$;
a single edge plus isolated vertices, or a balanced triangle plus
isolated vertices, gives $\{-1,1\}$; every other graph gives the full
field.  Indeed, two disjoint edges allow independent adjacent swaps,
whose correction changes generate all three values.  A graph without
such a matching is a star or a triangle, together with isolated
vertices.  A star with at least two edges gives all three values by
placing its leaves before or after the center.  Directly considering
the six orders of a triangle distinguishes the balanced and unbalanced
cases.  As an additional implementation check, the graph classifier is
compared with all permutations in $756$ generated subset cases; the
classification argument itself does not depend on that diagnostic check.

The second seven-term enumeration gives the same prefix counts and
ends with no surviving length-seven prefix.  The complete eight-term
results are compared below.  ``Coverage closures'' count entire
subtrees closed by the proved rank-and-coverage rule, rather than
visited complete leaves.
\begin{center}
\begin{tabular}{lrr}
\hline
 & Ascending implementation & Descending implementation\\
\hline
Visited leaves & $3325333$ & $2495319$\\
Deficient-rank leaves & $320$ & $320$\\
Coverage closures, length $6$ & $37831$ & $37831$\\
Coverage closures, length $7$ & $1287248$ & $1249091$\\
\hline
\end{tabular}
\end{center}
The descending prefix counts are
\[
 1,\ 8,\ 73,\ 1028,\ 14059,\ 172604,\ 1466883,\ 2495319.
\]
The late counts differ because different normalized prefix orders
close different hereditary subtrees.  Each search separately exhausts
its complete required search space; equality of late counts is not
needed. Both seven-term programs report complete exhaustion with a
full quotient-zero witness in every case; both eight-term programs
report complete exhaustion with the moment criterion satisfied in
every case. Their logs record these completion statuses explicitly.
Elapsed times are informational; exhaustive completion is established by
the stated searches and their termination statuses.

\subsection{Transfer to the group sequence}
\label{subsec:isotropic-transfer}

For clarity we state exactly the input bounds used in this transfer.
Let $E_2$ be the extraspecial group of exponent three and order $3^5$,
and let $G=E_2\times C_3^r$, with $r\geq1$.  Then
$P=G/Z(G)$ is the symplectic four-space above.  The image in $P$ of
a noncentral term will be called its active vector.

\begin{proposition}\label{prop:isotropic-cap-three}
Suppose a product-one-free sequence $S$ over $G$ has length $2r+11$
and exactly $2r-1$ central terms.  Assume that its twelve active vectors
satisfy the bounds
\begin{equation}\label{eq:isotropic-initial-caps}
 \text{line }\leq2,\qquad
 \text{isotropic plane }\leq5,\qquad
 \text{nondegenerate plane }\leq6,\qquad
 \text{hyperplane }\leq8.
\end{equation}
Then every isotropic plane contains at most three active occurrences.
\end{proposition}

\begin{proof}
First suppose that an isotropic plane $U$ contains five active
occurrences.  Its full inverse image $A\leq G$ is abelian, contains
$G'$, and satisfies
\[
 b=\dim_{\mathbb F_3}A=r+3,\qquad
 a=\dim_{\mathbb F_3}(G/A)=2.
\]
Here the group operation on $A$ agrees with addition in Lie
coordinates.  With $C=S\cap A$ and $R=S\setminus C$, we have
\[
 |C|=(2r-1)+5=2r+4=2b-2,\qquad |R|=7.
\]
The inherited line and plane bounds are those of
Lemma~\ref{lem:isotropic-seven}.  A hyperplane through $U$ already
contains the five inside active terms, so its outside occupancy is
at most $8-5=3$.  That lemma therefore supplies a full quotient-zero
outside block.  Apply Theorem~\ref{thm:relative} with defect $d=2$.
The moment degree is zero: every nonempty full block contributes the
row $(1)$, so the full-row space is all of $\mathbb F_3$.
The required exclusion of the triangle target from that space is
impossible.  Consequently the isotropic bound improves from five to
four.

Now suppose that $U$ contains four active occurrences.  For its same
abelian inverse image $A$, the counts become
\[
 |C|=(2r-1)+4=2r+3=2b-3,\qquad |R|=8.
\]
The outside sequence has line bound two, isotropic-plane bound four,
nondegenerate-plane bound six, and through-$U$ hyperplane bound
$8-4=4$.  Thus Lemma~\ref{lem:isotropic-eight} gives $W_U\in M_U$.

We relate the target in that lemma, which includes single-edge blocks,
to the triangle target of Theorem~\ref{thm:relative}.  Let $W_E$ and
$W_T$ be the signed augmented-row sums for single-edge and
balanced-triangle blocks, respectively, so $W_U=W_E+W_T$.
Choose quotient coordinates $\bar v_i=(\bar v_{i1},\bar v_{i2})$ on
$P/U$, and put
\[
 Q(X)=\prod_{j=1}^{2}\left(1-
             \left(\sum_i\bar v_{ij}X_i\right)^2\right),\qquad
 e(X)=\sum_{i<j:\,\omega(v_i,v_j)\ne0}X_iX_j.
\]
For $m(X)=1$ or $X_k$, the degree of $Qem$ is at most seven, below
the eight Boolean variables.  Its alternating Boolean sum is zero.
Combining these nine identities into an augmented-row identity, a
single-edge block contributes its signed row, a triangle contributes
zero because it has three edges, and a commuting block contributes
zero.  All other contributions lie in $M_U$.  Hence $W_E\in M_U$,
and therefore $W_T\in M_U$ as well.

Theorem~\ref{thm:relative}, now with defect $d=3$ and moment degree
one, excludes precisely this membership for a product-one-free
sequence.  This contradiction excludes four inside active terms and
proves the asserted bound three.
\end{proof}

The proposition concerns the layer with exactly $2r-1$ central terms.
In particular, its new conclusion does not by itself exclude thinner
central layers.  Its finite hypotheses and the transfer above make no
assumptions on the unspecified central coordinates of the original
terms.

\section{A potential theorem for twelve active terms}\label{sec:twelve}
Throughout this section $S=(q_1,\ldots,q_{12})$ is a sequence in
$P\setminus\{0\}$ satisfying the following bounds:
\begin{equation}\label{eq:fourcaps}
\begin{array}{c|cccc}
\text{subspace}&\text{line}&\text{isotropic plane}&
\text{nondegenerate plane}&\text{hyperplane}\\\hline
\text{occupancy}&\le2&\le3&\le6&\le8.
\end{array}
\end{equation}
These are exactly the bounds derived in the preceding sections for
a sequence in Theorem~\ref{thm:main}.

\subsection{The family of required triangle cores}
Let $\mathcal K(S)$ be the family of balanced triples of distinct
\emph{vector values} in $\supp(S)$ with the following property.
Writing $t=u+v+w$, include $\{u,v,w\}$ if
\begin{equation}\label{eq:completion}
t=0\quad\text{or, for }t\ne0,\quad
\operatorname{mult}_S(-t)\ge1\ \text{or}\
\operatorname{mult}_S(t)\ge2.
\end{equation}
Let $\mathcal B(S)$ be the graph whose vertices are the values in
$\supp(S)$ and whose edges are the pairs contained in at least three
members of $\mathcal K(S)$.

\begin{theorem}[Finite potential theorem]\label{thm:potential12}
Under~\eqref{eq:fourcaps}, the cochain $a_{uv}=\omega(u,v)$ on
$\mathcal B(S)$ is a gradient over $\F$. Consequently there are
symmetric weights $\rho_{uv}\in\F$ on the nonorthogonal value pairs
such that
\[
\rho_{uv}+\rho_{vw}+\rho_{wu}=1
\qquad(\{u,v,w\}\in\mathcal K(S)).
\]
\end{theorem}

We first justify the exact completion condition. For a balanced
triple, $t$ is orthogonal to its three vertices. If they are linearly
dependent, they span a nondegenerate plane, so $t=0$. If they are
independent, $t\ne0$ and
\[
\langle u,v,w\rangle^\perp=\F t.
\]
The isolated terms of a block with this triangle core lie on that
line. Its occupancy is at most two, so the sum $-t$ is attainable
precisely by one occurrence of $-t$ or two of $t$. These terms are
different from all the core vertices. If $t=0$, the triangle itself
is already a zero-sum block.

It is important to fold identical values before forming
$\mathcal B(S)$. A balanced triple cannot contain the same value
twice. Any solution on value pairs lifts to occurrence pairs by
copying the corresponding weight. Conversely, an edge of
$\mathcal B(S)$ has at least three \emph{different third values};
this fact is used in the frame reduction.

\subsection{An integral-cycle criterion}
\begin{lemma}\label{lem:branchpotential}
Let $\mathcal K$ be any family of balanced triangles in an alternating
$\F$-space. Let $\mathcal B$ consist of the edges lying in at least
three members of $\mathcal K$. If $\omega(u,v)=p(v)-p(u)$ on
$\mathcal B$ for a potential $p$, there are edge weights summing to
one on every member of $\mathcal K$.
\end{lemma}
\begin{proof}
Let $D$ be the unsigned edge-triangle incidence matrix. The system
$D^T\rho=\mathbf1$ is soluble if and only if every
$\lambda\in\ker_{\F}D$ satisfies $\sum_T\lambda_T=0$.
Order the vertices and orient all edges and triangles accordingly.
Lift $\lambda_T$ to $\{-1,0,1\}$ and the nonzero edge pairing $a_e$
to $\{-1,1\}$. On a balanced triangle $T$ we have
$\delta a(T)=3t_T$ for $t_T\in\{-1,1\}$.
For the integral two-chain $z=\sum_T t_T\lambda_T[T]$, the exact
boundary identity is
\[
\partial z=\operatorname{diag}(a_e)D\lambda.
\]
This boundary is divisible by three. On an edge contained in at most
two triangles, the corresponding entry of $D\lambda$ is an integer
between $-2$ and $2$, and therefore is zero. Thus
$\gamma=\partial z/3$ is an integral one-cycle supported on
$\mathcal B$. Integrally,
\[
\langle a,\gamma\rangle
=\frac13\langle\delta a,z\rangle
=\sum_T\lambda_T.
\]
Modulo three the left side is
$\langle\delta p,\gamma\rangle=\langle p,\partial\gamma\rangle=0$.
This proves the dual solvability condition.
\end{proof}

\subsection{Complete generation of the minimal frames}
If $\mathcal B(S)$ is empty, the potential assertion is immediate.
Otherwise choose a branching edge $u,v$, orient it so that
$\omega(u,v)=1$, and apply a symplectic transformation taking
$u=e_1$, $v=f_1$. Put $c=-u-v$ and
$H=\langle u,v\rangle^\perp\cong\F^2$.
Every balanced third vertex has the unique form
\begin{equation}\label{eq:wz}
w_z=c+z,\qquad z\in H.
\end{equation}
Indeed, writing $w=au+bv+z$, the two balance equations give
$a=b=-1$. Choose three distinct third vertices, hence three distinct
parameters in $H$, with $z=0$ allowed.

For each nonzero $z$, condition~\eqref{eq:completion} requires one
of the multisets $\{-z\}$ and $\{z,z\}$. For $z=0$ the required
multiset is empty. Combine requirements by their coordinatewise
maximum multiplicities, not their sum: the same occurrence may serve
as a companion for more than one core. Discard any such companion
multiset that strictly contains another valid requirement.
Together with $u,v$ and the three $w_z$, the remaining minimal
requirements are the \emph{frames}.

Every actual $S$ containing the chosen three cores contains one of
these minimal frames. All bounds in~\eqref{eq:fourcaps} are inherited
by subsequences. The subgroup fixing $u$ and $v$ acts on $H$ as
$\operatorname{SL}_2(3)$, of order 24, and acts trivially on
$\langle u,v\rangle$. Canonicalizing the complete frame multiset
under these 24 maps gives a lossless reduction.

The first generator examines all $\binom93=84$ parameter triples and
the at most $2^3$ choices of companion requirements. There are 508
distinct minimal choices before occupancy tests, 248 passing those
tests, and eleven orbits under the 24 maps. A second generator
instead examines, for each parameter triple, all $6^4=1296$
companion multisets with occupancy at most two on the four lines
of $H$. A requirement is minimal precisely when deleting any one
occurrence destroys at least one core's completion condition.
This different generation method gives the same 508, 248 and eleven
results, and exactly the same canonical frame list.

The codes in Table~\ref{tab:frames} use
$a_1+3a_2+9b_1+27b_2$ for $(a_1,a_2,b_1,b_2)$.
Neither generator imposes a bound on the number of distinct values
or on repetition excess. Geometrically, three different doubled
companions would, together with the three $w_z$, place nine
occurrences in $\langle c,H\rangle$, contradicting the hyperplane
bound eight.

\begin{table}[htbp]
\centering\small
\begin{tabular}{rl}
\toprule
Frame & Codes\\\midrule
0 & $1,3,9,20,26,27,74$\\
1 & $1,3,9,26,27,30,74,80$\\
2 & $1,3,9,26,27,50,60,74$\\
3 & $1,3,3,9,23,27,30,74,80$\\
4 & $1,3,3,9,23,27,50,60,74$\\
5 & $1,3,3,9,23,27,53,57,74$\\
6 & $1,3,3,9,23,47,50,54,60$\\
7 & $1,3,3,9,23,27,27,30,47,80$\\
8 & $1,3,3,9,23,27,27,33,47,77$\\
9 & $1,3,3,9,23,27,27,47,50,60$\\
10 & $1,3,3,9,23,27,27,47,53,57$\\\bottomrule
\end{tabular}
\caption{The eleven minimal branching frames.}\label{tab:frames}
\end{table}

\subsection{Complete extension and potential verification}
For each fixed frame, append every sorted multiset of nonzero values
needed to reach length twelve, retaining precisely those prefixes
that satisfy~\eqref{eq:fourcaps}. The unrestricted implementation
does not identify symmetric extensions and does not impose any
additional pruning condition. This search covers every admissible
multiset extending the frame. The actual required core family and
the whole branching graph are recomputed at every leaf.

To test the potential condition, choose a root in each connected
component, set its potential to zero, and propagate
$p(v)=p(u)+\omega(u,v)$ along edges. If both endpoint values are
already assigned, check this equality. Equality at every edge is
equivalent to the existence of a potential. All operations are exact
in $\F$.

\begin{table}[htbp]
\centering
\begin{tabular}{lrr}
\toprule
Frame & With frame stabilizer & Unrestricted\\\midrule
0 & 7\,453\,266 & 7\,453\,266\\
1 & 995\,486 & 995\,486\\
2 & 342\,022 & 995\,918\\
3--6, each & 24\,308 & 24\,308\\
7--10, each & 536 & 536\\\midrule
Total & 8\,890\,150 & 9\,544\,046\\\bottomrule
\end{tabular}
\caption{Completed twelve-term extension counts. Every leaf passes
the potential test.}\label{tab:extensions}
\end{table}

The second implementation independently generates the geometry,
uses the pairing of coordinates $(0,3),(1,2)$, converts the displayed
frames by interchanging the last two coordinates, and appends values
in descending rather than ascending order. It applies no symmetry
reduction. The two implementations complete all eleven frames with
the counts in Table~\ref{tab:extensions}. Every leaf passes the
potential test. No small-support theorem or Gaussian-solvability
fallback is used to close a leaf.

For completeness, the stabilizer pruning in the first implementation
is lossless. In each stabilizer orbit choose the lexicographically
least sorted \emph{complete} extension. If a sorted prefix were
made smaller by a stabilizer map, the sorted first entries of the
image of the complete extension would be no larger than the sorted
image of that prefix, and hence would make the full extension
smaller. Thus every orbit has a representative whose prefixes survive.
Only frame 2 has a nontrivial stabilizer affecting the displayed
counts.

Any configuration with a branching edge contains a frame after the
normalization above and is therefore covered by one of these
complete searches. Configurations without a branching edge already
have the potential property. This proves the potential assertion of
Theorem~\ref{thm:potential12}; Lemma~\ref{lem:branchpotential}
proves its edge-weight conclusion.

\section{The affine moment identity and the group contradiction}
\label{sec:conclusion}
For the twelve active occurrences, let $\mathcal F$ be the full
active-zero blocks and $\mathcal T$ the active-zero blocks consisting
of a balanced triangle and isolated vertices. Put
\[
r_1(B)=(1,\mathbf1_B(1),\ldots,\mathbf1_B(12)),\qquad
M_1=\Span_{\F}\{r_1(B):B\in\mathcal F\},
\]
\[
W_1=\sum_{B\in\mathcal T}(-1)^{|B|}r_1(B).
\]
The weights supplied by Theorem~\ref{thm:potential12}, lifted from
values to occurrences, give
\[
E(X)=\sum_{i<j,\,\omega(q_i,q_j)\ne0}\rho_{q_iq_j}X_iX_j,
\qquad
Q(X)=\prod_{\alpha=1}^4
\left(1-\left(\sum_iq_{i\alpha}X_i\right)^2\right).
\]
On the Boolean cube, $Q$ indicates active sum zero. For each
$L\in\{1,X_1,\ldots,X_{12}\}$, the polynomial $QEL$ has degree
at most $8+2+1=11<12$, so its alternating Boolean sum vanishes.

An active-zero block cannot have exactly one nonorthogonal edge:
pairing its sum with either endpoint would give a nonzero value.
By the ordering-correction classification, the remaining nonfull
blocks are commuting blocks or members of $\mathcal T$.
On these two families $E$ has values zero and one, respectively.
The thirteen coordinate identities therefore give
\begin{equation}\label{eq:finalmoment}
W_1=-\sum_{B\in\mathcal F}(-1)^{|B|}E(\mathbf1_B)r_1(B)
\in M_1.
\end{equation}

\begin{proof}[Proof of Theorem~\ref{thm:main}]
Suppose that $S$ were a product-one-free sequence as in the theorem.
The preceding subspace reductions give~\eqref{eq:fourcaps} for its
twelve active projections. Apply the relative moment theorem
(Theorem~\ref{thm:relative}) with $A=Z(G)$,
$b=r+1$, $a=4$ and $d=3$. Indeed,
\[
|S\cap A|=2r-1=2b-3,\quad
|S\setminus A|=12=2a+3+1,\quad D=d-2=1.
\]
It requires $W_1\notin M_1$, contrary to~\eqref{eq:finalmoment}.
\end{proof}

\section{Scope, verification and provenance}\label{sec:scope}
The twelve-term result treats exactly the central occupancy
$2r-1$ at total length $2r+11$. It is uniform in the number of central
direct factors, since the finite geometry is always $P=\F^4$ and
the coefficient degree remains three. Occupancies smaller than
$2r-1$ lead to higher-degree moment problems and are not covered.
In particular, no exact value of $\dc(E_2\times C_3^r)$ for all $r$
is asserted here.

The finite geometric theorem and its proof do not require the
stronger quadratic moment assertion for nine vectors outside an
isotropic plane. They use the already established occupancy three
on isotropic planes directly. Thus the proof does not assume the
occupancy-two improvement that would follow from that stronger
assertion.

The supplied files contain only the finite checks used here, their
inputs, full run records, and instructions. Counts refer to the
enumeration protocol of the corresponding program; in the twelve-term
case a sequence may extend more than one minimal frame. Counts from
different frames or implementations are not counts of globally distinct
isomorphism classes. A timeout, a partial log, or a successful sample
is not treated as a complete proof. The final checks report successful
completion only after all their prescribed branches have been exhausted.

\paragraph{Reproduction.}
The directory \file{anc/} contains eight standalone C++17 programs,
two Python frame generators, their inputs and complete logs.
Run \verb|python3 run_all.py --seconds 600| from that directory
with Python 3 and a C++17 compiler available. The time allowance
can be increased on slower hardware; reaching a time limit is a
failure, never a completed verification. The decisive programs are
\file{verify_twelve_primary.cpp} and
\file{verify_twelve_independent.cpp}; the frame generators are
\file{make_branch_frames_primary.py} and
\file{make_branch_frames_independent.py}.
The remaining sources are named \file{verify_line9_primary.cpp},
\file{verify_line9_independent.cpp}, \file{verify_plane7_primary.cpp},
\file{verify_plane7_independent.cpp}, \file{verify_plane8_primary.cpp}
and \file{verify_plane8_independent.cpp}.
The supplied \file{README.md} gives the exact expected counts, and
\file{SHA256SUMS} records the shipped ancillary snapshot. All
mathematical tests use exact arithmetic; elapsed-time reporting is
not part of the mathematical verification.

\paragraph{Use of AI-assisted tools.}
OpenAI's ChatGPT was used for proof development, algebraic checking,
design and implementation of the finite searches, literature searches,
and preparation of the manuscript. The separate implementations were
developed within this AI-assisted workflow; their separation refers to
code and algorithms, not to independent human peer review.
The computer-assisted claims are specified as exact finite assertions,
with their reductions, source programs and completed execution records
provided for inspection and reproduction.


\end{document}